\documentclass[10pt]{amsart}
\pdfoutput=1
\usepackage{amssymb}
\usepackage{amsmath}
\usepackage{amsthm}
\usepackage{mathrsfs}
\usepackage{array}
\usepackage{caption}
\usepackage{lscape}
\usepackage{color}
\usepackage{tikz-cd}

\usepackage{spectralsequences}

\include{scrload}

\usepackage{hyperref}  
\hypersetup{%
  bookmarksnumbered=true,%
  bookmarks=true,%
  colorlinks=true,%
  linkcolor=blue,%
  citecolor=blue,%
  filecolor=blue,%
  menucolor=blue,%
  pagecolor=blue,%
  urlcolor=blue,%
  pdfnewwindow=true,%
  pdfstartview=FitBH}

\def\makeautorefname#1#2{\expandafter\def\csname#1autorefname\endcsname{#2}}
\makeautorefname{section}{Section}%
\makeautorefname{theorem}{Theorem}%
\makeautorefname{thm}{Theorem}%
\makeautorefname{addm}{Addendum}%
\makeautorefname{mainthm}{Main theorem}%
\makeautorefname{corollary}{Corollary}%
\makeautorefname{const}{Construction}%
\makeautorefname{cor}{Corollary}%
\makeautorefname{lemma}{Lemma}%
\makeautorefname{notn}{Notation}%
\makeautorefname{lem}{Lemma}%
\makeautorefname{sublemma}{Sublemma}%
\makeautorefname{sublem}{Sublemma}%
\makeautorefname{subl}{Sublemma}%
\makeautorefname{prop}{Proposition}%
\makeautorefname{warn}{Warning}%
\makeautorefname{eg}{Example}%
\makeautorefname{figure}{Figure}

\newtheorem{thm}{Theorem}[section]
\newtheorem{cor}{Corollary}[section]
\newtheorem{prop}{Proposition}[section]
\newtheorem{lem}{Lemma}[section]

\theoremstyle{definition}

\newtheorem{rem}{Remark}[section]

\newenvironment{pf}{\begin{proof}}{\end{proof}}

\makeatletter
\let\c@cor=\c@thm
\let\c@prop=\c@thm
\let\c@lem=\c@thm
\let\c@defn=\c@thm
\let\c@eg=\c@thm
\let\c@notn=\c@thm
\let\c@rem=\c@thm
\let\c@warn=\c@thm
\let\c@sch=\c@thm
\let\c@equation=\c@thm
\makeatother

\newcommand{\cA}{\ensuremath{\mathcal{A}}}
\newcommand{\C}{\ensuremath{\mathbb{C}}}

\newcommand{\F}{\ensuremath{\mathbb{F}}}

\newcommand{\bM}{\ensuremath{\mathbb{M}_2}}

\newcommand{\bP}{\ensuremath{\mathbb{P}}}

\newcommand{\R}{\ensuremath{\mathbb{R}}}

\newcommand{\Z}{\ensuremath{\mathbb{Z}}}

\newcommand{\rtarr}{\longrightarrow}

\newcommand{\xrtarr}[1]{\xrightarrow{#1}}

\newcommand{\iso}{\cong}

\newcommand{\cl}{\mathrm{cl}}
\DeclareMathOperator{\Ext}{Ext}
\DeclareMathOperator{\Tor}{Tor}

\newcommand{\hzeromult}{
\class(\lastx,\lasty+1)
\structline
}

\newcommand{\hzerotower}[1]{
    \foreach \i in {2,...,#1} {
        \hzeromult
    }
}

\newcommand{\honemult}{
\class(\lastx+1,\lasty+1)
\structline
}

\newcommand{\rhodiv}{
\class(\lastx+1,\lasty)
\structline[red]
}

\newcommand{\rhomult}{
  \class(\lastx-1,\lasty)
  \structline[red]
}

\newcommand{\rhocotower}[1]{
    \foreach \i in {2,...,#1} {
        \rhodiv
    }
}

\NewSseqGroup{\rhodiamond}{m}{
  \class[#1](0,0)
  \rhomult
  \honemult
  \structline[dashed](\lastx,\lasty,-1)(\lastx,\lasty-1,-1)
  \rhodiv
  \structline(\lastx,\lasty,-1)(\lastx-1,\lasty-1,-1)
}

\newcommand{\rhotowerhonediv}[2]{
    \foreach \i in {2,...,#1} {
        \rhomult
        \ifnum \numexpr\i-1 < #2
          \structline(\lastx,\lasty,-1)(\lastx-1,\lasty-1,-1)
        \fi
    }
}

\newcommand{\rhocotowerhonediv}[2]{
    \foreach \i in {2,...,#1} {
        \rhodiv
        \ifnum \numexpr\i-1 < #2
          \structline(\lastx,\lasty,-1)(\lastx-1,\lasty-1,-1)
        \fi
    }
}

\NewSseqGroup{\hiddenhzero}{}{
    \structline(0,0,-1)(0,1,-1)
}

\NewSseqGroup{\hzeroline}{}{
    \structline(0,0,-1)(0,1,-1)
}

\definecolor{magenta}{rgb}{1 0 1}

\NewSseqGroup{\hzerotauline}{}{
    \structline[color=magenta](0,0,-1)(0,1,-1)
}

\NewSseqGroup{\hiddenhone}{}{
    \structline(0,0,-1)(1,1,-1)
}

\NewSseqGroup{\honeline}{}{
    \structline(0,0,-1)(1,1,-1)
}

\NewSseqGroup{\htwoline}{}{
    \structline(0,0,-1)(3,1,-1)
}

\NewSseqGroup{\htwotauline}{}{
    \structline[color=magenta](0,0,-1)(3,1,-1)
}

\NewSseqGroup{\rholine}{}{
    \structline[red](0,0,-1)(-1,0,-1)
}

\NewSseqGroup{\hzeroloc}{}{
  \savestack
  \class(0,1)
  \restorestack
  \hzeroline(0,0)
  \draw[->,>=stealth](0,1,-1)--(0.0,1.7);
}

\NewSseqGroup{\rholoc}{}{
  \draw[red,->,>=stealth](0,0)--(-0.7,0);
}

\NewSseqGroup{\honeloc}{}{
  \draw[semithick,red,->,>=stealth](0,0)--(0.75,0.75);
}

\newcommand{\ExtMWLabel}[1] {
  \draw[fill=white](0.5,12.4) rectangle (7.5,13.6);
  \node[font=\large] at (4,13) { \mathrm{Ext_{C_2}}, \text{coweight } #1};
}

\newcommand{\EinfMWLabel}[1] {
  \draw[fill=white](-0.2,12.4) rectangle (10.2,13.6);
  \node[font=\large] at (5,13) {\mathrm{Adams} \ E_\infty\ \text{for}\  S^0, \text{coweight } #1};
}

\newcommand{\EinfkoMWLabel}[1] {
  \draw[fill=white](0.5,11.4) rectangle (9.5,13.6);
  \node[font=\large] at (5,13) {\mathrm{Adams} \ E_\infty\ \text{for}\  ko_{C_2},};
  \node[font=\large] at (5.0,12.1) {\text{coweight } #1};
}

\newcolumntype{C}{>{$}c<{$}}
\newcolumntype{L}{>{$}l<{$}}
\newcolumntype{R}{>{$}r<{$}}
\newcolumntype{H}{>{\setbox0=\hbox\bgroup$}c<{$\egroup}@{}}

\numberwithin{equation}{section}

\begin{document}


\sseqset{
    classes= {fill, inner sep = 0pt, minimum size = 0.25em},
    class labels={below=0.3pt,font=\scriptsize}, 
    differentials=black,
    struct lines=semithick,
    class pattern=linear, 
    class placement transform = { rotate = 135, scale = 1 },
    run off differentials = ->, 
    grid = go, 
    grid step=2,
    grid color = gray!50!white,
    scale=0.5,
    y range={0}{14}, 
    x range={0}{20}, 
    y tick step = 2,
    x tick step = 2,
}

\newcommand{\n}{x}

\renewcommand{\n}{0}
\begin{sseqdata}[name=MW\n]

\draw[dashed,fill=black!20] (0,-1) -- (21.0,9.5) -- (21.0,-1);

\begin{scope}[blue]
\class(0,0)
\rholoc(0,0)
\foreach \i in {1,...,14} {
  \honemult
  \rholoc(\lastx,\lasty)
}
\honemult
\hzeroloc(0,0)
\end{scope}

\class["{Q}\!h_1^4"](5,3)
\DoUntilOutOfBounds{
\honemult
}

\foreach \i in {1,2} {
\class(5+\i,3)
\structline[red](\lastx,\lasty)(\lastx-1,\lasty)
\DoUntilOutOfBounds{
\honemult
\structline[red](\lastx,\lasty)(\lastx-1,\lasty)
}
}

\class(9,4)
\structline[red](\lastx,\lasty)(\lastx-1,\lasty)
\DoUntilOutOfBounds{
\honemult
\structline[red](\lastx,\lasty)(\lastx-1,\lasty)
}

\foreach \i in {0,1,2} {
\class(13+\i,7)
\structline[red](\lastx,\lasty)(\lastx-1,\lasty)
\DoUntilOutOfBounds{
\honemult
\structline[red](\lastx,\lasty)(\lastx-1,\lasty)
}
}

\class(17,8)
\structline[red](\lastx,\lasty)(\lastx-1,\lasty)
\DoUntilOutOfBounds{
\honemult
\structline[red](\lastx,\lasty)(\lastx-1,\lasty)
}

\foreach \i in {0,1,2} {
\class(21+\i,11)
\structline[red](\lastx,\lasty)(\lastx-1,\lasty)
\DoUntilOutOfBounds{
\honemult
\structline[red](\lastx,\lasty)(\lastx-1,\lasty)
}
}

%

\class(7,1)
\rhodiv
\honemult
\rhomult
\honeline(\lastx-1,\lasty-1)
\hiddenhzero(\lastx,\lasty-1)

\class(6,2)
\rhocotower{3}

\class(8,3)
\rhodiv
\hiddenhone(7,2)
\hiddenhone(8,2)
\hiddenhzero(8,2)

\class(14,4)
\savestack
\honemult
\savestack\rhodiv\restorestack
\honemult
\rhocotowerhonediv{2}{2}
\restorestack
\hiddenhzero(16,5)
\hzerotower{3}
\rhocotower{3}

\class(15,1)
\rhodiv
\honemult
\rhomult
\honeline(\lastx-1,\lasty-1)
\hiddenhzero(\lastx,\lasty-1)

\class(16,7)
\hiddenhzero(\lastx,\lasty-1)
\rhodiv
\foreach \i in {-1,0} {\hiddenhone(\lastx-1+\i,\lasty-1) }

\class(18,2)
\hzeromult
\structline(\lastx,\lasty,-1)(\lastx-1,\lasty-1)

\class(14,2)
\hzeromult
\rhocotower{5}

\class(18,4)
\hiddenhzero(\lastx,\lasty-1)

\class(19,3)
\rhodiv
\draw[red](20,3)--(20.75,3);


\class(20,4)
\hzeromult


\ExtMWLabel{\n}

\end{sseqdata}


\renewcommand{\n}{0}
\begin{sseqdata}[name=MW{\n}Einf]

\begin{scope}[blue]
\class(0,0)
\DoUntilOutOfBounds{
  \rholoc(\lastx,\lasty)
  \honemult
}
\hzeroloc(0,0)
\end{scope}

\class["{Q}\!h_1^4"](5,3)
\DoUntilOutOfBounds{
  \structline[red,dashed](\lastx,\lasty)(\lastx-1,\lasty+1)
  \honemult
}

\foreach \i in {1,2} {
\class(5+\i,3)
\DoUntilOutOfBounds{
  \structline[red](\lastx,\lasty)(\lastx-1,\lasty)
  \honemult
}
}

\class(9,4)
\structline[red](\lastx,\lasty)(\lastx-1,\lasty)
\DoUntilOutOfBounds{
\honemult
\structline[red](\lastx,\lasty)(\lastx-1,\lasty)
}

\foreach \i in {0,1,2} {
\class(13+\i,7)
\structline[red](\lastx,\lasty)(\lastx-1,\lasty)
\DoUntilOutOfBounds{
\honemult
\structline[red](\lastx,\lasty)(\lastx-1,\lasty)
}
}

\class(17,8)
\structline[red](\lastx,\lasty)(\lastx-1,\lasty)
\DoUntilOutOfBounds{
\honemult
\structline[red](\lastx,\lasty)(\lastx-1,\lasty)
}

\foreach \i in {0,1,2} {
\class(21+\i,11)
\structline[red](\lastx,\lasty)(\lastx-1,\lasty)
\DoUntilOutOfBounds{
\honemult
\structline[red](\lastx,\lasty)(\lastx-1,\lasty)
}
}

%
%

\draw[dashed,fill=black!20] (0,-1) -- (21.0,9.5) -- (21.0,-1);

\class(7,1)
\rhodiv
\honemult
\rhomult
\honeline(\lastx-1,\lasty-1)
\hzeroline(\lastx,\lasty-1)

\class(6,2)
\rhocotower{3}

\class(8,3)
\rhodiv
\honeline(7,2)
\honeline(8,2)
\hzeroline(8,2)

\class(18,2)
\hzeromult

\class(14,4)

\class(14,2)
\class(15,3)
\structline[red,dashed](\lastx,\lasty,-1)(\lastx-1,\lasty+1,-1)
\rhocotower{3}

\class(15,6)
\rhodiv

\class(15,1)
\structline[red,dashed](\lastx,\lasty)(\lastx-1,\lasty+1)
\rhodiv
\honemult
\honeline(\lastx,\lasty)
\rhomult
\honeline(\lastx-1,\lasty-1)
\hzeroline(\lastx,\lasty-1)

\class(16,7)
\hzeroline(\lastx,\lasty-1)
\rhodiv
\foreach \i in {-1,0} {\honeline(\lastx-1+\i,\lasty-1) }

\class(19,3)
\rhodiv
\draw[red](20,3)--(20.75,3);

\class(20,4)
\hzeromult


\EinfMWLabel{\n}

\end{sseqdata}


\renewcommand{\n}{0}
\begin{sseqdata}[name=MW{\n}Einfko, class placement transform = { rotate = -45, scale = 1 }]

\begin{scope}[blue]
\class(0,0)
\DoUntilOutOfBounds{
  \rholoc(\lastx,\lasty)
  \honemult
}
\hzeroloc(0,0)
\end{scope}

\class["{Q}\!h_1^3"](4,2)
\DoUntilOutOfBounds{
  \structline[red,dashed](\lastx,\lasty)(\lastx-1,\lasty+1)
  \honemult
}
\hzeroloc(4,2)

\foreach \i in {1,2,3} {
  \class(5+\i,3)
  \DoUntilOutOfBounds{
    \structline[red](\lastx,\lasty)(\lastx-1,\lasty)
    \honemult
  }
}
\hzeroloc(8,3)

\class(12,6)
\structline[red](\lastx,\lasty)(\lastx-1,\lasty)
\DoUntilOutOfBounds{
\honemult
\structline[red](\lastx,\lasty)(\lastx-1,\lasty)
}
\hzeroloc(12,6)

\foreach \i in {1,2,3} {
\class(13+\i,7)
\structline[red](\lastx,\lasty)(\lastx-1,\lasty)
\DoUntilOutOfBounds{
\honemult
\structline[red](\lastx,\lasty)(\lastx-1,\lasty)
}
}
\hzeroloc(16,7)

\class(20,10)
\structline[red](\lastx,\lasty)(\lastx-1,\lasty)
\DoUntilOutOfBounds{
\honemult
\structline[red](\lastx,\lasty)(\lastx-1,\lasty)
}
\hzeroloc(20,10)


\EinfkoMWLabel{\n}

\end{sseqdata}

\title[$C_2$-equivariant stable homotopy groups]{The Bredon-Landweber region in $C_2$-equivariant stable homotopy groups}
\author{Bertrand J. Guillou}
\address{Department of Mathematics\\ University of Kentucky\\
Lexington, KY 40506, USA}
\email{bertguillou@uky.edu}

\author{Daniel C. Isaksen}
\address{Department of Mathematics\\ Wayne State University\\
Detroit, MI 48202, USA}
\email{isaksen@wayne.edu}
\thanks{The first author was supported by NSF grant DMS-1710379.
The second author was supported by NSF grant DMS-1202213.}



\begin{abstract}
We use the $C_2$-equivariant Adams spectral sequence 
to compute part of the $C_2$-equivariant
stable homotopy groups $\pi^{C_2}_{n,n}$.
This allows us to recover results of Bredon and Landweber
on the image of the geometric fixed-points map
$\pi^{C_2}_{n,n} \rightarrow \pi_0$.  We also recover
results of Mahowald and Ravenel on the Mahowald root invariants of 
the elements $2^k$.
\end{abstract}

\date{\today}

\maketitle

\section{Introduction}
\label{sctn:intro}

The goal of this article is to study some phenomena in the
$C_2$-equivariant stable homotopy groups.
Let $\R^{n,k}$ be the $n$-dimensional real representation of $C_2$
in which the nonidentity element of
$C_2$ acts as $-1$ on the last $k$ coordinates (and trivially on the first $n-k$),
and let $S^{n,k}$ be its one-point compactification.
Then $\pi^{C_2}_{n,k}$ is the set of 
$2$-completed $C_2$-equivariant
stable homotopy classes of maps $S^{n,k} \rtarr S^{0,0}$.
In this article, we are primarily concerned with the groups
$\pi^{C_2}_{k,k}$.

The classical Hopf map $\eta_\cl:S^3\rtarr S^2$ can be modeled as the defining quotient map $\C^2-\{0\} \rtarr \C\bP^1$ for complex projective space. When we remember the action of $C_2$ via complex conjugation, this represents a $C_2$-equivariant stable map $\eta$ in $\pi^{C_2}_{1,1}$. 
Classically, $\eta_\cl$ is nilpotent in the stable homotopy ring, as is every element in positive stems \cite{N}. 
However, the equivariant Hopf map $\eta$ is not nilpotent
because $\eta$ induces the non-nilpotent element $-2$ on 
geometric fixed points.
(The distinction between $2$ and $-2$ depends on choices of 
orientations and is inconsequential to the argument.)
We will concern ourselves with phenomena associated to the non-zero 
elements $\eta^k$ in $\pi^{C_2}_{k,k}$.

Because the fixed points of the representation sphere $S^{k,k}$ consist
of two points, the geometric fixed point homomorphism takes the form
$\phi:\pi_{k,k}^{C_2} \rtarr \pi_0\iso \Z$.
Bredon \cite{Br} and Landweber \cite{L} proved that the image of $\phi$
is not in general generated by $\phi(\eta^k)$. For instance, $\phi(\eta^5) = (-2)^5=-32$, 
but $\phi(\pi^{C_2}_{5,5})=16\Z$. 
In fact, the higher powers of $\eta$ are increasingly divisible by 2 
in the $C_2$-equivariant stable homotopy groups (\autoref{etadiv2}).

Let $\rho: S^{0,0} \rtarr S^{1,1}$ be the inclusion of fixed points.
Our main result describes to what extent the powers of $\eta$ are divisible by $\rho$, from which we will deduce several other results.

\begin{thm}\label{mainthm} 
Let $k = 4j + \varepsilon \geq 1$, where $0 \leq \epsilon \leq 3$.
If $\varepsilon = 0$, then
the $C_2$-equivariant stable homotopy class $\eta^k$ is divisible by $\rho^{k-1}$ and no higher power of $\rho$. 
Otherwise, the $C_2$-equivariant stable homotopy class $\eta^{k}$ is divisible by $\rho^{4j}$ and no higher power of $\rho$. 
\end{thm}

Our primary tool for studying $C_2$-equivariant
stable homotopy groups is the equivariant Adams spectral sequence
\cite{G} \cite{HK}.
The {\bf Bredon-Landweber region} refers to 
the subgroups of $\pi^{C_2}_{k,k}$ that are detected in 
Adams filtration greater than $\frac{1}{2}k - 1$.
This region is displayed in the top part of \autoref{AdamsEinfChart}.
We will show that the Bredon-Landweber region 
is additively generated by
the elements $\rho^i \eta^k$, together
with elements $\alpha$ such that
$\rho^i \alpha = \eta^k$ for some $i$.

We recover the following theorem of Landweber \cite[Theorem~2.2]{L}, which was originally conjectured by Bredon \cite{Br}. 

\begin{cor} 
\label{cor:fixed-pt-image}
Let $k = 8j + \varepsilon \geq 1$, with $0 \leq \varepsilon \leq 7$.
The image of the geometric fixed points homomorphism $\pi^{C_2}_{k,k}\xrtarr{\phi} \pi_0$ is generated by 
\[\begin{cases}
2^{4j + 1} & \text{if} \quad \varepsilon = 0. \\
2^{4j + \varepsilon} & \text{if} \quad 1 \leq \varepsilon \leq 4. \\
2^{4j + 4} & \text{if} \quad 5 \leq \varepsilon \leq 7. \\
\end{cases}\]
\end{cor}

Similarly, we have:

\begin{cor}\label{etadiv2} 
Let $k = 8j + \varepsilon \geq 5$, with $0 \leq \varepsilon \leq 7$.
The $C_2$-equivariant stable homotopy class $\eta^k$ is divisible by 
\[ \begin{cases} 
2^{4j-1} & \text{if} \quad \varepsilon = 0, \\ 
2^{4j} & \text{if} \quad 1 \leq \varepsilon \leq 4, \\
2^{4j+\varepsilon-4} & \text{if} \quad 5 \leq \varepsilon \leq 7, \\
\end{cases} \]
and no higher power of $2$.
\end{cor}

The proofs of \autoref{mainthm}, \autoref{cor:fixed-pt-image},
and \autoref{etadiv2} appear in \autoref{sctn:Adams}.
First, we must carry out some $C_2$-equivariant Adams spectral
sequence calculations.

Our calculations can also be used to compute the classical
Mahowald invariants of $2^k$ for all $k \geq 0$ 
(see \autoref{thm:Mahowald-inv}).
We directly apply the Bruner-Greenlees formulation \cite{BG}
of the Mahowald invariant that uses $C_2$-equivariant homotopy groups.
These invariants were previously established by
Mahowald and Ravenel \cite{MR} using entirely different methods.

The charts were created using Hood Chatham's {\tt spectralsequences} package.

\section{Notation}
\label{sec:background}

We continue with notation from \cite{DI} and \cite{GHIR} as follows.
\begin{enumerate}
\item
$\bM^\R=\F_2[\tau,\rho]$ 
is the motivic cohomology of $\R$ with $\F_2$ coefficients, where $\tau$ and $\rho$ have bidegrees $(0,1)$ and $(1,1)$, respectively.
\item
$\bM^{C_2}$ is the bigraded $C_2$-equivariant Bredon cohomology of a point with coefficients in the constant Mackey functor $\underline{\F}_2$. 
\item
$\cA^{C_2}$ is the 
$C_2$-equivariant mod 2 Steenrod algebra, using coefficients $\underline{\F}_2$, and
$\cA^{C_2}(1)$ is the $\bM^{C_2}$-subalgebra of $\cA^{C_2}$
generated by $\mathrm{Sq}^1$ and $\mathrm{Sq}^2$.
\item
$\Ext_\cl$, $\Ext_\C$, $\Ext_\R$, and $\Ext_{C_2}$ are the
cohomologies of the classical, $\C$-motivic, $\R$-motivic, and
$C_2$-equivariant mod 2 Steenrod algebras respectively.
These objects are the $E_2$-pages of Adams spectral sequences.
\item
$\pi_{*,*}^{C_2}$ and $\pi_*$ are the stable homotopy rings of the
2-completed $C_2$-equivariant sphere spectrum and the
2-completed classical sphere spectrum respectively.
\end{enumerate}

We will use some specific familiar elements of the Adams $E_2$-page.
These elements lie near the ``Adams edge" at the top of the Adams
chart along a line of slope $1/2$.
Our notation for these elements is standard.  They include
$P^k h_1$, $P^k h_1^2$, $P^k h_1^3$, $P^k c_0$, $P^k h_1 c_0$,
$P^k h_2$, and $P^k h_0 h_2$.  
In addition, we will consider the elements
$P^k h_0 h_3$, $P^k h_0^2 h_3$, and $P^k h_0^3 h_3$.
These slightly non-standard (but technically correct) 
names conveniently refer to a well-understood,
regular family of elements in the Adams $E_2$-page.
They are the top three elements in a tower of $h_0$-multiplications
in stems congruent to $7$ modulo $8$.
For more details,
we refer the reader to any Adams chart, 
such as \cite{Icharts} or \cite[Figure A3.1]{R}.

We follow \cite{Istems} in grading $\Ext$ groups
according to  $(s,f,w)$, where:
\begin{enumerate}
\item
$f$ is the Adams filtration, i.e., the homological degree.
\item
$s+f$ is the internal degree, i.e., corresponds to the
first coordinate in the bidegree of the Steenrod algebra.
\item
$s$ is the stem, i.e., the internal degree minus
the Adams filtration.
\item
$w$ is the weight.
\end{enumerate}

Following this grading convention,
the elements $\tau$ and $\rho$, as elements of $\Ext_\R$, have degrees
$(0,0,-1)$ and $(-1,0,-1)$ respectively.
We will also often refer to the \textbf{coweight}, 
which is defined to be $c = s-w$. 
Since both $\eta$ and $\rho$ have coweight $0$, 
the Bredon-Landweber region consists entirely of elements of
coweight $0$. The coweight is also called the Milnor-Witt degree in the motivic context (\cite{DI},\cite{GI}). 

\section{The $\rho$-Bockstein spectral sequence}
\label{sctn:rho-Bock}

As an $\bM^\R$-module, 
the equivariant cofficient ring $\bM^{C_2}$ splits as 
$\bM^{C_2} \iso \bM^\R \oplus NC$, where $NC$ is the ``negative cone''. 
The negative cone has $\F_2$-basis $\{ \frac{\gamma}{\rho^j \tau^{k+1}}\}$, where $j,k\geq 0$ and $\frac{\gamma}{\tau}$ lives in degree $(0,0,2)$.
See \cite[Section 2.1]{GHIR} for more details.
This splitting of $\bM^{C_2}$ induces a splitting
\[
\Ext_{C_2} \iso \Ext_\R \oplus \Ext_{NC},
\]
where $\Ext_{NC}=\Ext_\R(NC,\bM^\R)$.
The splitting of $\bM^{C_2}$
also yields a splitting for the Bockstein spectral sequence, and we
follow \cite[Proposition 3.1]{GHIR} in writing $E_1^+$ for the summand of the Bockstein $E_1$-term which converges to $\Ext_\R$ and $E_1^-$ for the summand  which converges to $\Ext_{NC}$.

The $\R$-motivic $\rho$-Bockstein spectral sequence (\cite{H,DI}) takes the form
\[ E_1^+ = \Ext_\C[\rho] \Rightarrow \Ext_\R.\]
The groups $\Ext_\R$ are computed for low coweights in \cite{DI}. In coweight $0$, $E_1^+$ is $\F_2[h_0,h_1,\rho]/(h_0h_1)$, and the only relevant Bockstein differential is $d_1(\tau) = \rho h_0$, giving 

\begin{lem}[\cite{DI}]
\label{lem:ExtR-mw0}
 $\Ext_\R$ in coweight 0 is $\F_2[h_0,h_1,\rho]/(h_0h_1,\rho h_0)$.
\end{lem}

The calculation of $\Ext_{NC}$ in coweight $0$ is much
more complicated.
We have a short exact sequence \cite[Prop 3.1]{GHIR}
\begin{equation}
\bigoplus_{s\geq 0} \frac{\F_2[\tau]}{\tau^\infty} 
\left\{\frac\gamma{\rho^s}\right\} \otimes_{\F_2[\tau]} \Ext_{\C}
\rightarrow E_1^- \rightarrow
\bigoplus_{s\geq 0} \Tor_{\F_2[\tau]}
\left( \frac{\F_2[\tau]}{\tau^\infty} \left\{\frac\gamma{\rho^s}\right\},
\Ext_{\C}\right),
\label{E1NCses}
\end{equation}
which we abbreviate as
\[ \gamma E_1^- \rtarr E_1^- \rtarr Q E_1^-.\]
As explained in \cite[Remark~3.5]{GHIR}, 
for each class $x$ in $\Ext_\C$ such that $\tau x = 0$, 
we get a class $Q x$ in $Q E_1^-$, and this element is infinitely
divisible by $\rho$.
The $\tau$-torsion elements of $\Ext_\C$ in coweight 0 are $h_1^k$ for $k\geq 4$, so these give rise to infinitely $\rho$-divisible classes $Qh_1^k$.  This describes $QE_1^-$ in coweight $0$.

We now describe $\gamma E_1^-$ in coweight $0$.
First note that $\frac{\gamma}{\tau^i}$ has coweight
$-i-1$.
Now let $x$ be a class in $\Ext_\C$ that is $\tau$-free
and not divisible by $\tau$, and let $c \geq 0$ be the coweight
of $x$.
If $c \geq 2$, 
then $\frac{\gamma}{\tau^{c-1}} x$ is an element of $\gamma E_1^-$
in coweight $0$
that is infinitely divisible by $\rho$.
When $c \leq 1$, there is no corresponding
element of $\gamma E_1^-$ in coweight $0$.

This description of $E_1^-$ is incomplete in the sense that it depends
on the $\tau$-free part of $\Ext_\C$, which is only known in a range.
The $\tau$-free part of $\Ext_\C$ corresponds precisely to 
classical $\Ext_{\cl}$ \cite[Proposition 2.10]{Istems}.  
In a range, information about
$\Ext_{\cl}$ can be obtained from an $\Ext$ chart, such as
\cite{Icharts} or \cite[Figure A.3.1]{R}.

In order to rule out certain Bockstein differentials later, we need some
structural results for the Bockstein spectral sequence.

\begin{prop}
\label{prop:rho-cofree-diff}
Let $x$ be an element of $E_r^-$ such that
$d_r(x)$ is non-zero.  Then $x$ and $d_r(x)$ are both
infinitely divisible by $\rho$ in $E_r^-$.
\end{prop}

\begin{proof}
Let $E_r^-[k]$ be the part of $E_r^-$ in Bockstein filtration
$k$.  Note that $E_1^-[k]$ is zero if $k > 0$, and that
$\rho: E_1^-[k] \rtarr E_1^-[k+1]$ is an isomorphism if $k < 0$.
The $d_r$ differential takes the form
$E_r^-[k] \rtarr E_r^-[k+r]$.

By induction,
diagram chases show that 
$\rho: E_r^-[k] \rtarr E_r^-[k+1]$ is injective if
$0 > k > -r$, and it is an isomorphism if $-r \geq k$.
In particular, if $-r \geq k$, then every element
of $E_r^-[k]$ is infinitely divisible by $\rho$.

Now let $x$ be an element of $E_r^-[k]$ such that $d_r(x)$ is
non-zero.  This implies that $E_r^-[k+r]$ is non-zero, so
$-r \geq k$, and $x$ is infinitely divisible by $\rho$.
Finally, the multiplicative structure implies that $d_r(x)$ 
must also be infinitely divisible by $\rho$.
\end{proof}

\begin{rem}
\label{rem:rho-cofree-diff}
\autoref{prop:rho-cofree-diff} is dual to
\cite[Lemma 3.4]{DI}, which shows that if $d_r(x)$ is
non-zero in $E_r^+$, then $\rho^k x$ and $\rho^k d_r(x)$ 
are non-zero in $E_r^+$ for all $k \geq 0$.
In fact, the proof of \autoref{prop:rho-cofree-diff} dualizes
line by line.
\end{rem}

\begin{prop}\label{NegConeVanish} 
The $C_2$-equivariant Bockstein $E_1$-page is zero 
in degrees $(s,f,w)$ such that
the coweight $s - w$ is negative,
the stem $s$ is positive, and 
$f > \frac{1}{2} s + \frac{3}{2}$.
\end{prop}

\begin{pf}
The summand $E_1^+$ vanishes when $s - w < 0$, i.e., in negative
coweights.
Similarly, $QE_1^-$ vanishes in negative coweights.

It only remains to consider $\gamma E_1^-$.
Consider a non-zero element of $\gamma E_1^-$ in degree $(s,f,w)$
with $s > 0$.
This element has the form 
$\frac{\gamma}{\rho^j \tau^k} x$, where $x$ is $\tau$-free in $\Ext_\C$.
Moreover, the degree of $x$ is $(s - j, f, w - j - k - 1)$.

Using a vanishing result for $\Ext_\C$ 
\cite[Theorem~1.1]{GIVanish}, 
we know that
\[
f \leq \frac{1}{2} (s-k) + \frac{3}{2}.
\]
Since $k$ is non-negative, it follows that
$f \leq \frac{1}{2} s + \frac{3}{2}$.
\end{pf}

\begin{lem}\label{h1PerNC1} 
In coweight $1$, the
localization $E_1[h_1^{-1}]$ of the Bockstein $E_1$-page
vanishes.
\end{lem}

\begin{pf}
We know that $\Ext_\C[h_1^{-1}]$ vanishes in
coweight $1$ \cite[Theorem 1.1]{GI}.
Therefore $E_1^+[h_1^{-1}]$ and $QE_1^-[h_1^{-1}]$ both vanish.
Finally, $\gamma E_1^- [h_1^{-1}]$ also vanishes
because there are no $\tau$-free classes in
$\Ext_\C[h_1^{-1}]$.
\end{pf}

\section{Some Bockstein differentials}

The goal of this section is to compute some explicit Bockstein
differentials that we will need for our analysis of the
Bredon-Landweber region.

\begin{lem}
\label{lem:gamma-diff}
For $k \geq 0$,
\[
d_1\left( \frac{\gamma}{\rho \tau^{2k+1}} \right) = 
\frac{\gamma}{\tau^{2k+2}} h_0.
\]
\[
d_2\left( \frac{\gamma}{\rho^2 \tau^{4k+2}} \right) =
\frac{\gamma}{\tau^{4k+3}} h_1.
\]
\[
d_3\left( \frac{\gamma}{\rho^3 \tau^{4k+4}} \right) = 0.
\]
\end{lem}

\begin{proof}
These formulas follow from the Leibniz rule and the
$\R$-motivic Bockstein differentials
$d_1(\tau^{2k+1}) = \rho \tau^{2k} h_0$, 
$d_2(\tau^{4k+2}) = \rho^2 \tau^{4k+1} h_1$,
and $d_3(\tau^{4k+4}) = 0$
\cite[Proposition 3.2]{DI}.

More specifically, start with the relation
$\tau^{2k+1} \cdot \frac{\gamma}{\rho \tau^{2k+1}} = 0$.
Apply the $d_1$ differential to obtain
\[
0 = 
\rho \tau^{2k} h_0 \cdot \frac{\gamma}{\rho \tau^{2k+1}} + 
\tau^{2k+1} \cdot d_1 \left( \frac{\gamma}{\rho \tau^{2k+1}} \right) =
\frac{\gamma}{\tau} h_0 + 
\tau^{2k+1} \cdot d_1 \left( \frac{\gamma}{\rho \tau^{2k+1}} \right).
\]
Therefore,
$d_1 \left( \frac{\gamma}{\rho \tau^{2k+1}} \right)$ must equal
$\frac{\gamma}{\tau^{2k+2}} h_0$.

The second and third formulas follow from a similar argument, starting with
the relations $\tau^{4k+2} \cdot \frac{\gamma}{\rho^2 \tau^{4k+2}} = 0$
and $\tau^{4k+4} \cdot \frac{\gamma}{\rho^2 \tau^{4k+4}} = 0$.
\end{proof}

\begin{lem}
\label{lem:v1-perm}
For $k \geq 0$, the elements
$\tau P^k h_1$, $P^k h_2$, and $\tau P^k c_0$, are
permanent cycles in the $\R$-motivic $\rho$-Bockstein
spectral sequence.
\end{lem}

\begin{proof}
We can express the classes $\tau P^k h_1$ recursively
as matric Massey products \cite{Quig}:
\[
\tau P^k h_1 = \left\langle [ h_3 \  c_0 ], \begin{bmatrix} h_0^4 \\ \rho^3 h_1^2\end{bmatrix}, \tau P^{k-1} h_1\right\rangle.
\]
The May Convergence Theorem \cite[Theorem~4.1]{May} 
\cite[Theorem~2.2.1]{Istems},
applied to the $\rho$-Bockstein spectral sequence, 
shows that $\tau P^k h_1$ is a permanent cycle.

Similarly, we have recursive matric Massey products
\[
P^k h_2 = \left\langle [ h_3 \  c_0 ], \begin{bmatrix} h_0^4 \\ \rho^3 h_1^2\end{bmatrix}, P^{k-1} h_2 \right\rangle
\]\[
\tau P^k c_0 = \left\langle [ h_3 \  c_0 ], \begin{bmatrix} h_0^4 \\ \rho^3 h_1^2\end{bmatrix}, \tau P^{k-1} c_0\right\rangle.
\]
\end{proof}

\begin{lem}
\label{lem:d3}
For $k \geq 0$,
\[
d_3(\tau^3 P^k h_0^3 h_3) = \rho^3 \tau P^{k+1} h_1.
\]
\[
d_3(\tau^3 P^k h_1 c_0) = \rho^3 P^{k+1} h_2.
\]
\end{lem}

\begin{proof}
By \cite[Theorem~4.1]{DI}, the only classes in $\Ext_\R[\rho^{-1}]$ that survive the $\rho$-inverted Bockstein spectral sequence are those satifsying $s+f-2w=0$.
Since $\tau P^k h_1$  does not satisfy this equation, 
either it supports a Bockstein differential, or
$\rho^r \tau P^k h_1$  is hit by a Bockstein differential for some $r$.
But \autoref{lem:v1-perm} shows that $\tau P^k h_1$ does not
support a differential.
Therefore, $\rho^r \tau P^k h_1$ must be hit by some differential.
By inspection, there is only one possibility.
This establishes the first formula.

The same argument applies to establish the second formula.
\end{proof}

\begin{lem}\label{d4km1Qh14}
For $k\geq 1$,
\[d_{4k-1}\left( \frac{Q}{\rho^{4k-1}} h_1^{4k}\right) = \frac{\gamma}{ \tau^{4k-1}} P^{k-1}h_0^3 h_3.\]
\[d_{4k}\left( \frac{Q}{\rho^{4k}} h_1^{4k+1}\right) = \frac{\gamma}{ \tau^{4k}} P^{k}h_1.\]
\end{lem}

\begin{pf}
The class $ \frac{Q}{\rho^{4k}} h_1^{4k}$ restricts to a class of the same name in $\Ext$ over $\cA(1)^{C_2}$. There, 
we have 
\[
d_{4k}\left( \frac{Q}{\rho^{4k}} h_1^{4k}\right) = 
\frac{\gamma}{\tau^{4k}} b^k
\]
by \cite[Proposition~7.9]{GHIR}. 
The second formula follows immediately from this, since
$P^k h_1$ restricts to $h_1 b^k$.

On the other hand, 
the Bockstein class $\frac{\gamma}{\tau^{4k}} b^k$ is not in the image of the restriction from the $\rho$-Bockstein spectral sequence over $\cA^{C_2}$,
so $\frac{Q}{\rho^{4k}} h_1^{4k}$ must support a shorter Bockstein differential over $\cA^{C_2}$. The claimed differential is the only possibility.
\end{pf}

\autoref{tbl:BockDiff} summarizes the key differential calculations
that we will need later.

\begin{table}[ht]
\captionof{table}{Some Bockstein differentials}
\label{tbl:BockDiff}
\begin{center}
\begin{tabular}{LLLLLl} 
\hline
c &  (s,f, w) & \text{element} & r & d_r & proof \\ \hline  
-2k-2 & (1, 0, 2k+3) & \frac{\gamma}{\rho \tau^{2k+1}} & 1 &
	\frac{\gamma}{\tau^{2k+2}} h_0 & \autoref{lem:gamma-diff} \\
-4k-3 & (2, 0, 4k+5) & \frac{\gamma}{\rho^2 \tau^{4k+2}} & 2 &
	\frac{\gamma}{\tau^{4k+3}} h_1 & \autoref{lem:gamma-diff} \\
4k+6 & (8k+7, 4k+4, 4k+1) & \tau^3 P^k h_0^3 h_3 & 3 
	& \rho^3 \tau P^{k+1} h_1 & \autoref{lem:d3} \\
4k+6 & (8k+9, 4k+4, 4k+3) & \tau^3 P^k h_1 c_0 & 3 
	& \rho^3 P^{k+1} h_2 & \autoref{lem:d3} \\
0 & (8k, 4k-1, 8k) & \frac{Q}{\rho^{4k-1}} h_1^{4k} & 4k-1 
  & \frac{\gamma}{\tau^{4k-1}}P^{k-1}h_0^3 h_3 & \autoref{d4km1Qh14} \\
0 & (8k+2, 4k, 8k+2) & \frac{Q}{\rho^{4k+1}} h_1^{4k+1} & 4k
  & \frac{\gamma}{\tau^{4k}}P^kh_1 & \autoref{d4km1Qh14} \\
\hline
\end{tabular}
\end{center}
\end{table}

\section{$\Ext_{C_2}$ in coweight $0$}

\autoref{prop:ExtC2-BL} explicitly describes a large part
of $\Ext_{C_2}$ in coweight $0$.  This result is more
easily understood in the $\Ext$ chart in \autoref{ExtChart},
where we are considering only elements above the shaded region.

\begin{prop}
\label{prop:ExtC2-BL}
In degrees $(s,f,w)$ satisfying $s- w = 0$ and 
$f > \frac{1}{2} s - 1$,
$\Ext_{C_2}$ consists of the following classes:
\begin{enumerate}
\item
$h_0^k$ for $k \geq 0$.
\item
$\rho^j h_1^k$ for $j \geq 0$ and $k \geq 0$.
\item
$\frac{Q}{\rho^j} h_1^{4k + \varepsilon}$,
with $k \geq 1$, $0 \leq \varepsilon \leq 3$ and:
\begin{enumerate}
\item
$j \leq 4k - 2$ when $\varepsilon = 0$.
\item
$j \leq 4k - 1$ when $1 \leq \varepsilon \leq 3$.
\end{enumerate}
\end{enumerate}
\end{prop}

\begin{proof}
\autoref{lem:ExtR-mw0} explains how the classes
$h_0^k$ and $\rho^j h_1^k$ arise in $\Ext_\R$.
It remains to study $\Ext_{NC}$.

The desired elements of the form $\frac{Q}{\rho^j}{h_1^{4k+\varepsilon}}$
arise from the differentials in \autoref{d4km1Qh14}.
\autoref{prop:rho-cofree-diff} implies that 
these elements cannot be involved in any further differentials.

There are several additional Adams periodic families of elements in the
Bockstein $E_1$-page that lie above the line $f = \frac{1}{2}s-1$
in coweight $0$.  All of these elements are the 
targets of Bockstein differentials, 
as shown in \autoref{tbl:BockDiff-MW1},
so they do not appear in $\Ext_{C_2}$.
Each differential in \autoref{tbl:BockDiff-MW1} follows from the
Leibniz rule and the differentials in \autoref{tbl:BockDiff}.

However, the last three calculations are not entirely obvious.
In these cases, write 
\[
\frac{\gamma}{\rho^2 \tau^{4k+1}} P^k c_0 =
\frac{\gamma}{\rho^2 \tau^{4k+2}} \cdot \tau P^k c_0
\]
\[
\frac{\gamma}{\rho^3 \tau^{4k+1}} P^k h_0^3 h_3 =
\frac{\gamma}{\rho^3 \tau^{4k+4}} \cdot \tau^3 P^k h_0^3 h_3
\]
\[
\frac{\gamma}{\rho^3 \tau^{4k+1}} P^k h_1 c_0 =
\frac{\gamma}{\rho^3 \tau^{4k+4}} \cdot \tau^3 P^k h_1 c_0,
\]
and then apply the Leibniz rule to these products.

\begin{table}[ht]
\captionof{table}{Some Bockstein differentials in coweight $1$}
\label{tbl:BockDiff-MW1}
\begin{center}
\begin{tabular}{LLLLL} 
\hline
c &  (s,f, w) & \text{element} & r & d_r \\ 
\hline
1 & (8k+3, 4k+1, 8k+2) & \frac{\gamma}{\rho^2\tau^{4k-2}}P^k h_1 & 2 &
	\frac{\gamma}{\tau^{4k-1}} P^k h_1^2 \\
1 & (8k+4, 4k+1, 8k+3) & \frac{\gamma}{\rho\tau^{4k-1}}P^k h_2 & 1 &
	\frac{\gamma}{\tau^{4k}} P^k h_0 h_2 \\
1 & (8k+4, 4k+2, 8k+3) & \frac{\gamma}{\rho\tau^{4k-1}}P^k h_0h_2 & 1 & 
	\frac{\gamma}{\tau^{4k-1}} P^k h_1^3 \\
1 & (8k+8, 4k+2, 8k+7) & \frac{\gamma}{\rho\tau^{4k+1}} P^k h_0 h_3 & 1 &
	\frac{\gamma}{\tau^{4k+2}} P^k h_0^2h_3 \\
1 & (8k+8, 4k+3, 8k+7) & \frac{\gamma}{\rho\tau^{4k+1}} P^k h_0^2 h_3 
	& 1 & \frac{\gamma}{\tau^{4k+2}} P^k h_0^3h_3 \\
1 & (8k+10, 4k+3, 8k+9) & \frac{\gamma}{\rho^2\tau^{4k+1}} P^k c_0 & 2 &
	\frac{\gamma}{\tau^{4k+2}} P^k h_1 c_0 \\
1 & (8k+10, 4k+4, 8k+9) & \frac{\gamma}{\rho^3\tau^{4k+1}} P^k h_0^3h_3 
	& 3 & \frac{\gamma}{\tau^{4k+3}} P^{k+1} h_1 \\
1 & (8k+12, 4k+4, 8k+11) & \frac{\gamma}{\rho^3\tau^{4k+1}}P^k h_1c_0 
	& 3 & \frac{\gamma}{\tau^{4k+4}} P^{k+1} h_2 \\
\hline
\end{tabular}
\end{center}
\end{table}
\end{proof}

\autoref{ExtChart} also shows some 
classes that are not part of the Bredon-Landweber region. 
These classes arise in the shaded part of the chart.
The structure there is quite complicated, and it will be analyzed
in a range in future work.

\section{The Adams spectral sequence}
\label{sctn:Adams}

We  show in \autoref{NoAdamsDiffs} 
that the entire Bredon-Landweber region described in
\autoref{prop:ExtC2-BL} survives 
the $C_2$-equivariant Adams spectral sequence.

\begin{prop}\label{NoAdamsDiffs} No element
listed in \autoref{prop:ExtC2-BL} 
is either the target or the source of an Adams differential.
\end{prop}

\begin{pf}
Except for the elements $h_0^k$, 
all of the classes in the Bredon-Landweber region are $h_1$-periodic. Therefore, any class supporting an Adams differential into the Bredon-Landweber region would be $h_1$-periodic and of coweight 1. 
\autoref{h1PerNC1} shows that there are no such classes.

On the other hand,
the elements $h_0^k$ are $h_0$-periodic.
By inspection in low dimensions, there are no
$h_0$-periodic elements that could support differentials whose
values are $h_0^k$.

Adams differentials on the classes in the Bredon-Landweber region 
lie in the vanishing region of \autoref{NegConeVanish}.
Therefore, these classes must be permanent cycles.
\end{pf}

\autoref{AdamsEinfChart} shows the Bredon-Landweber region
in the Adams $E_\infty$-page.
Similarly to the Adams $E_2$-page in \autoref{ExtChart}, there are additional classes 
in the shaded part of the chart that we will consider in future work.

We now analyze hidden $\rho$ extensions in the Bredon-Landweber region.
The key tool is \autoref{prop:rho-div}.

\begin{prop}
\label{prop:rho-div}
The kernel of the underlying homomorphism
$U: \pi^{C_2}_{n,k} \rtarr \pi_n$ is 
the image of $\rho:\pi^{C_2}_{n-1,k-1} \rtarr \pi^{C_2}_{n,k}$.
\end{prop}

\begin{proof}
This follows immediately from the 
cofiber sequence 
\[
(C_2)_+ \rtarr S^{0,0} \xrtarr{\text{ }\rho\text{ }} S^{1,1},
\]
using the free-forgetful adjunction between
equivariant homotopy classes $(C_2)_+ \rtarr X$ and
classical homotopy classes from $S^0$ to the
underlying spectrum of $X$.
\end{proof}

\begin{lem}
\label{lem:rho-hidden}
For $k \geq 4$,
there is a hidden $\rho$ extension from
$Q h_1^k$ to $h_1^k$.
\end{lem}

\begin{proof}
The classical Hopf map $\eta_\cl$ in $\pi_1$ 
is the underlying map of the equivariant Hopf map
$\eta$ in $\pi^{C_2}_{1,1}$.
Let $k \geq 4$.
Since $(\eta_\cl)^k=0$ in $\pi_k$,
\autoref{prop:rho-div} implies that $\eta^k$ must be 
a multiple of $\rho$.
The only possibility is that there is a hidden
$\rho$ extension from $Q h_1^k$ to $h_1^k$.
\end{proof}

The hidden extensions of \autoref{lem:rho-hidden} appear in
\autoref{AdamsEinfChart} as dashed lines of negative slope.

In principle, it would be possible for there to be additional
hidden $\rho$ extensions whose sources lie in the
shaded part of \autoref{AdamsEinfChart} and whose targets
lie in the Bredon-Landweber region.  
\autoref{lem:no-rho-extn} eliminates this possibility.

\begin{lem} 
\label{lem:no-rho-extn}
There are no hidden $\rho$-extensions 
whose targets lie in the Bredon-Landweber region. 
\end{lem}

\begin{pf}
We use the unit map $S^{0,0} \rtarr ko_{C_2}$ for the $C_2$-equivariant connective real $K$-theory spectrum, as studied in \cite{GHIR}.
The entire Bredon-Landweber region is detected
by the Adams $E_\infty$-page for $ko_{C_2}$.
Therefore, a hidden $\rho$ extension into the
Bredon-Landweber region would be detected by a
$\rho$ extension in the homotopy of $ko_{C_2}$.

There are in fact some elements in the homotopy of 
$ko_{C_2}$ that support $\rho$ extensions into the image of
the Bredon-Landweber pattern.
These elements are detected by
$\frac{Q}{\rho^{4k-4}} h_1^{4k-1}$ and
$\frac{Q}{\rho^{4k-1}} h_1^{4k}$ in the
Adams $E_\infty$-page for $ko_{C_2}$.
We must show that they do not lie in the image
of the unit map.

These elements support infinite towers of $h_0$-multiplications
in the Adams $E_\infty$-page for $ko_{C_2}$.
Therefore, they cannot lie in the image of the unit map, since
the Adams $E_\infty$-page in \autoref{AdamsEinfChart}
does not include elements that support infinite towers
of $h_0$-multiplications in the relevant degrees.
\end{pf}

We have now collected enough results to prove \autoref{mainthm},
\autoref{cor:fixed-pt-image}, and \autoref{etadiv2}.

\begin{proof}[Proof of \autoref{mainthm}]
The $C_2$-equivariant element $\eta^k$ is detected by $h_1^k$ in the
Adams $E_\infty$-page.
\autoref{lem:rho-hidden} and \autoref{prop:ExtC2-BL} give a lower bound on the
power of $\rho$ that divides $\eta^k$, and
\autoref{lem:no-rho-extn} gives an upper bound on this
power of $\rho$.
\end{proof}

\begin{proof}[Proof of \autoref{cor:fixed-pt-image}]
The geometric fixed points homomorphism 
$\phi$ takes the values
$\phi (\rho) = 1$ and $\phi(\eta) = -2$.
(The minus sign in $\phi(\eta)$ depends on choices of orientations
and is inconsequential to the proof.)
Consequently, $\phi(\alpha) = 0$ if $\alpha$ is not $\rho$-periodic,
i.e., if $\alpha$ is detected below the Bredon-Landweber region.

The corollary now follows from \autoref{mainthm}.
\end{proof}

\begin{proof}[Proof of \autoref{etadiv2}]
Recall \cite[Section 8]{DI} that $h_0$ detects $2+\rho\eta$. 
Therefore, for homotopy classes detected by $h_0$-torsion classes,
multiplication by $2$ coincides with multiplication by $-\rho\eta$. 
Thus $\eta^k$ is divisible by $2^m$ if and only if 
$\eta^{k-m}$ is divisible by $\rho^m$.
This latter condition can be determined by \autoref{mainthm}.
\end{proof}

\section{The Mahowald invariant of $2^k$}

The goal of this section is compute the Mahowald invariant of $2^k$
for all $k \geq 0$.
We begin by determining the values of the underlying homomorphism
$U: \pi^{C_2}_{s,s} \rtarr \pi_s$
on classes in the Bredon-Landweber region.
\autoref{prop:rho-div} implies that $U(\alpha)$ is necessarily zero
when $\alpha$ is divisible by $\rho$.  On the other hand,
$U(\alpha)$ is always non-zero when $\alpha$ is not divisible by $\rho$.

\begin{thm}
\label{thm:U}
The underlying homomorphism $U: \pi^{C_2}_{s,s} \rtarr \pi_s$
takes values as described in \autoref{tbl:u0}.
\end{thm}

\begin{table}[ht]
\captionof{table}{Some values of the underlying homomorphism}
\label{tbl:u0}
\begin{center}
\begin{tabular}{LLL} 
\hline
s & \alpha \text{ detected by} & U(\alpha) \text{ detected by} \\
0 & 1 & 1 \\
1 & h_1 & h_1 \\
2 & h_1^2 & h_1^2 \\
3 & h_1^3 & h_1^3 \\
8k-1 & \frac{Q}{\rho^{4k-2}} h_1^{4k} & P^{k-1} h_0^3 h_3 \\
8k+1 & \frac{Q}{\rho^{4k-1}} h_1^{4k+1} & P^k h_1 \\
8k+2 & \frac{Q}{\rho^{4k-1}} h_1^{4k+2} & P^k h_1^2 \\
8k+3 & \frac{Q}{\rho^{4k-1}} h_1^{4k+3} & P^k h_1^3 \\
\hline
\end{tabular}
\end{center}
\end{table}

\begin{proof}
The underlying homomorphism $U$ induces a map of Adams spectral sequences $\Ext_{C_2} \rtarr \Ext_{\cl}$. 
This map of spectral sequences detects
the first four values in \autoref{tbl:u0}.

The summand $\Ext_{NC}$ lies in the kernel of this map
of spectral sequences.
Therefore,
if $\alpha$ in $\pi^{C_2}_{n,k}$ is detected by an element
of $\Ext_{NC}$, then $U(\alpha)$ must be detected in strictly
higher Adams filtration.
For each of the last four entries in \autoref{tbl:u0},
there is only one possible element in higher Adams filtration.
\end{proof}

The underlying homomorphism $U$ plays a central role in the Mahowald root invariant \cite{MR}. 
The Bruner-Greenlees  formulation of the
Mahowald invariant of a homotopy class is given as follows \cite{BG}.
 First, recall that the geometric fixed points homomorphism 
$\phi:\pi^{C_2}_{n,k} \rtarr \pi_{n-k}$ gives rise to an isomorphism 
\cite[Proposition 7.0]{AI}
\[ 
\pi^{C_2}_{*,*} [\rho^{-1}] \iso
\pi_* [\rho^{\pm 1}].
\]
Then  the Mahowald invariant is defined via the diagram
\[
\begin{tikzcd}
\pi_* \arrow[r] 
\arrow[dr, dashed, bend right=5] \arrow[ddr,  bend right=10, "M(-)"'] & \pi^{C_2}_{*,*}[\rho^{-1}] \\
 & \pi^{C_2}_{*,*} \arrow[u] \arrow[d, "U"] \\
  & \pi_*
\end{tikzcd}\]
Here the dashed arrow picks out an element
from the largest possible stem.
More precisely, given a classical stable homotopy class
$\alpha$ in $\pi_k$,
one first chooses an equivariant stable homotopy class $\beta$ 
in $\pi^{C_2}_{n, n-k}$ 
such that $\phi(\beta) = \alpha$ and such that $n$ is as large
as possible.  In particular, $\beta$ is not divisible by $\rho$,
for otherwise $n$ would not be maximal.
Then $M(\alpha)$ contains the element $U(\beta)$.
Beware that there can be more than one choice for $\beta$,
so the Mahowald invariant has indeterminacy in general.

Our $C_2$-equivariant calculations allow us to easily
recover the Mahowald invariants  of $2^k$ (\cite[Theorem~2.17]{MR}).

\begin{thm} 
\label{thm:Mahowald-inv}
Let $k = 4j + \varepsilon \geq 4$ with $0 \leq \varepsilon \leq 3$.
The Mahowald invariant of $2^k$ contains
an element that is detected by
\begin{center}
\begin{tabular}{ll}
$P^{j-1} h_0^3 h_3$ & if $\varepsilon = 0$. \\
$P^j h_1$ & if $\varepsilon = 1$. \\
$P^j h_1^2$ & if $\varepsilon = 2$. \\
$P^j h_1^3$ & if $\varepsilon = 3$. 
\end{tabular}
\end{center}
\end{thm}

\begin{pf}
\autoref{mainthm} determines the value of $\beta$ in the 
Bruner-Greenlees formulation of the Mahowald invariant.
Then \autoref{thm:U} gives the value of $U(\beta)$.
\end{pf}

\begin{rem}
\label{rem:MI-indet}
The indeterminacy in $M(2^k)$ is determined by the values
of the underlying map on classes that are detected in the
shaded region of \autoref{AdamsEinfChart}.  
It does not seem possible to predict the indeterminacy
of $M(2^k)$ in general.  However, inspection in low dimensions
shows that the indeterminacy of
$M(2^5)$ is generated by 
elements detected by $h_1^2 h_3$ and $h_1 c_0$,
while $M(2^k)$ has no indeterminacy for all other values of 
$k \leq 8$.
This indeterminacy calculation depends on a detailed analysis of the shaded
region of \autoref{AdamsEinfChart} and will be justified in
future work.
\end{rem}

\section{Charts}

Here is a key for reading the charts of \autoref{ExtChart}, \autoref{AdamsEinfChart}, and \autoref{AdamsEinfkoChart}:
\begin{enumerate}
\item
Blue dots indicate copies of $\F_2$ from $\Ext_\R$ (or $\Ext$ 
over the $\R$-motivic version of $\cA(1)$ 
in the case of \autoref{AdamsEinfkoChart}).
\item
Gray dots indicate copies of $\F_2$  from $\Ext_{NC}$ (or 
from the negative cone part of $\Ext$ over $\cA^{C_2}(1)$
in the case of \autoref{AdamsEinfkoChart}).
\item
Horizontal lines indicate multiplications by $\rho$.
\item
Dashed lines of negative slope indicate $\rho$ extensions that 
are hidden in the Adams spectral sequence.
\item
Vertical lines indicate multiplications by $h_0$. 
\item
Vertical arrows indicate infinite sequences of multiplications
by $h_0$.
\item
Lines of slope $1$ indicate multiplications by $h_1$.
\end{enumerate}



\begin{figure}[h]
\caption{}
\label{ExtChart}
\printpage[name=MW0]
\end{figure}

\begin{figure}
\caption{}
\label{AdamsEinfChart}
\printpage[name=MW{0}Einf]
\end{figure}
\medskip

\begin{figure}
\caption{}
\label{AdamsEinfkoChart}
\printpage[name=MW{0}Einfko]
\end{figure}

\end{document}